\documentclass[10pt]{amsart}
\usepackage{amssymb}
\usepackage{amsfonts}
\usepackage{ifthen}
\usepackage{amsthm}
\usepackage{amsmath,mathrsfs}
\usepackage{graphicx}
\usepackage{amscd,amssymb,amsthm}
\usepackage{color}
\newcounter{minutes}
\divide\time by 60
\newcounter{hours}
\multiply\time by 60 \addtocounter{minutes}{-\time}

\newtheorem{propo}{ Proposition}
\newtheorem{lemma}{Lemma}

\newtheorem{theorem}{Theorem}
\newtheorem{remark}{Remark}

\newtheorem{theorem a}{Theorem A}
\newtheorem{theorem b}{Theorem B}
\newtheorem{theorem C}{Theorem C}

\begin{document}
\title [New type integral inequalities  for convex functions and applications]{New type integral
inequalities for convex functions with applications II\\}%

	\author[K. Mehrez, P. Agarwal]{Khaled Mehrez and  Praveen Agarwal}

\address{Khaled Mehrez. D\'epartement de Math\'ematiques ISSAT Kasserine,
Universit\'e de Kairouan, Tunisia.}
\email{k.mehrez@yahoo.fr}
\address{Praveen Agarwal. Department of mathematics, Anand International college of engineering,Jaipur,Rajasthan, India}
\email{goyal.praveen@gmail.com}
\begin{abstract} We have recently established some integral inequalities for convex
functions via the Hermite-Hadamard's inequalities. In continuation here, we also establish some interesting new integral inequalities for convex functions
via the  Hermite--Hadamard's inequalities and Jensen's integral inequality. Useful applications involving special means are also included.
\end{abstract}
\def\thefootnote{}
\footnotetext{ \texttt{File:~\jobname .tex,
          printed: \number\year-0\number\month-\number\day,
          \thehours.\ifnum\theminutes<10{0}\fi\theminutes}
} \makeatletter\def\thefootnote{\@arabic\c@footnote}\makeatother
\maketitle
\vspace{.2cm}
\noindent{\textbf{ Keywords:}} Hermite--Hadamard  inequality,  Integral inequalities, Convex functions, Special means.\\
\\
\noindent \textbf{Mathematics Subject Classification (2010)}: 26D15, 26D10.
\section{Introduction}

Let $f:I\subseteq\mathbb{R}\longrightarrow\mathbb{R}$
is a convex function, $a , b\in I$ with $a < b,$ if and only if,
\begin{equation}\label{1}
f\left(\frac{a+b}{2}\right)\leq\frac{1}{b-a}\int_a^b f(x)dx\leq\frac{f(a)+f(b)}{2},
\end{equation}
is well  known  in  the  literature  as the  Hermite--Hadamard  inequality  for  convex  function. A vast literature related to (\ref{1}) have been produced by a large number of
mathematicians \cite{9} since it is considered to be one of the most famous inequality for
convex functions due to its usefulness and many applications in various branches
of Pure and Applied Mathematics, such as Numerical Analysis \cite{4}, Information
Theory \cite{2}, Operator Theory \cite{7} and others.

Very recenty, authors \cite{KA} established some new type integral inequalities for convex function via the Hermite--Hadamard  inequality. This paper is a continuation of some line of authors results in \cite{KA}. Motivated by above work here, we proved some interesting new type integral inequalities for differentiable convex functions by using the  Hermite--Hadamard  inequality and Jensen integral inequality. As applications, we obtain some new inequality involving special means of real numbers.\\

In the proof of the main results we will need the following two lemmas.

\begin{lemma}\cite{D}\label{l3}Let $f: I\subseteq\mathbb{R}\longrightarrow\mathbb{R}$ be  a  differentiable
 mapping  on $I^0,$ and $a,b\in I$ with $a<b$, then we have
\begin{equation}
\frac{f(b)+f(a)}{2}-\frac{1}{b-a}\int_a^bf(x)dx=\frac{(b-a)^2}{2}
\int_0^1 t(1-t)f^{\prime\prime}(ta+(1-t)b)dt.
\end{equation}
\end{lemma}
\begin{lemma}\label{l1}\cite{J}(Jensen inequality)
Let $\mu$ be a probability measure and let $\varphi\geq0$ be a convex function. Then, for all $f$ be a integrable function we have
\begin{equation}\label{k1}
\int \varphi\circ f d\nu\geq \varphi\left(\int f d\nu\right).
\end{equation}
\end{lemma}

\section{Main results}
\setcounter{equation}{0}

Now we are ready to present our main results asserted by Theorems \ref{th0} to \ref{th5}.

\begin{theorem}\label{th0}Let $f:I\subseteq\mathbb{R}\longrightarrow\mathbb{R}$
be a  differentiable  mapping  on $I^{0}$ with $a<b.$ If $|f^\prime|$ is convex and increasing on $[a,b]$, then the following inequality
\begin{equation}
\left|\frac{1}{b-a}\int_a^b f(x)dx-\frac{bf(b)-af(a)}{b-a}\right|\leq \frac{|a||f^\prime(a)|+|b||f^\prime(b)|}{2}.
\end{equation}
\end{theorem}
\begin{proof}Using integration by parts, which is verified under the conditions given
 in the theorem, we have
\begin{equation}\label{00}
\left|\frac{1}{b-a}\int_a^b f(t)dt-\frac{bf(b)-af(a)}{b-a}\right|\leq \left(
\frac{1}{b-a}\int_a^b|t||f^\prime(t)|dt\right).
\end{equation}
0n the other hand, using the fact that the functions $|f^\prime(t)|$ is convex and increasing on $[a,b]$ and the  $|t|$ is convex and increasing on $[a,b]$, thus the function $|t||f^\prime(t)|$ is also convex on $[a,b]$, as product of positive convex and increasing functions. Now, by the right hand side inequality (\ref{1}) we deduce that inequality (\ref{00}) is valid.
\end{proof}

\begin{theorem}\label{thma}Let $p>1,q\geq1$ and $f:I\subseteq\mathbb{R}\longrightarrow\mathbb{R}$
be a  differentiable  mapping  on $I^{0}$ with $a<b.$ If $|f^\prime|^q$ is convex and increasing, then the following inequality
\begin{equation}\label{tyt}
\left|\frac{1}{b-a}\int_a^b f(x)dx-\frac{bf(b)-af(a)}{b-a}\right|\leq \frac{\left(|f^\prime(b)|+|f^\prime(a)|\right)^{\frac{p-1}{p}}\left(|b|^p+|a|^p\right)^{\frac{q-1}{qp}}
.\left(|b|^p|f^\prime(b)|^q+|a|^p|f^\prime(a)|^q\right)^{\frac{1}{qp}}}{2^{2-1/p}}.
\end{equation}
\end{theorem}
\begin{proof}
Applying the $p>1$ on the inequality (\ref{00}), we  have
\begin{equation}
\left|\frac{1}{b-a}\int_a^b f(t)dt-\frac{bf(b)-af(a)}{b-a}\right|^p\leq \left(
\frac{1}{b-a}\int_a^b|t||f^\prime(t)|dt\right)^p.
\end{equation}
Now, we set $\varphi(t)=|t|^p,\;f(t)=t$ and $d\mu(t)=\frac{|f^\prime(t)|dt}{\int_a^b|f^\prime(t)|dt}.$ So, by means of Lemma \ref{l1},
 we get
\begin{equation}\label{KP1}
\left|\frac{1}{b-a}\int_a^b f(t)dt-\frac{bf(b)-af(a)}{b-a}\right|^p\leq
 \left(\frac{\int_a^b|f^\prime(t)|dt}{b-a}\right)^{p-1}.\left(\frac{1}{b-a}\int_a^b |t|^p|f^\prime(t)|dt\right)
\end{equation}
By the  power--mean  inequality, we have
\begin{equation}\label{KP4}
\frac{1}{b-a}\int_a^b |t|^p|f^\prime(t)|dt\leq
 \frac{1}{b-a}\left(\int_a^b|t|^p dt\right)^{1-\frac{1}{q}}.
\left(\int_a^b|t|^p |f^\prime(t)|^q dt\right)^{\frac{1}{q}}
\end{equation}
Since the functions $|f^\prime(t)|^q$ is convex on $[a,b]$ and the  $|t|^p$ is convex
 on $[a,b]$, for each $p>1,$ thus the function $|t|^p|f^\prime(t)|^q$ is also convex on $[a,b]$,
 as product of positive convex functions. By again of the right hand side inequality (\ref{1}), we have
\begin{equation}\label{KP5}
\frac{1}{b-a}\int_a^b|t|^p |f^\prime(t)|^q dt\leq\frac{|b|^p|f^\prime(b)|^q+|a|^p|f^\prime(a)|^q}{2}
\end{equation}
and
\begin{equation}\label{KP6}
\frac{1}{b-a}\int_a^b|t|^p dt\leq \frac{|b|^p+|a|^p}{2}.
\end{equation}
According to (\ref{KP4}), (\ref{KP5}) and (\ref{KP6}), we have
\begin{equation}\label{KP7}
\frac{1}{b-a}\int_a^b |t|^p|f^\prime(t)|dt\leq\frac{\left(|b|^p+|a|^p\right)^{1-\frac{1}{q}}
.\left(|b|^p|f^\prime(b)|^q+|a|^p|f^\prime(a)|^q\right)^{\frac{1}{q}}}{2}.
\end{equation}
Again, from the right hand side of inequality (\ref{1}), we have
\begin{equation}\label{KP3}
\frac{1}{b-a}\int_a^b|f^\prime(t)|dt\leq \frac{|f^\prime(b)|+|f^\prime(a)|}{2}.
\end{equation}
In view of (\ref{KP1}), (\ref{KP3}) and (\ref{KP7}) we obtain the desired result.
\end{proof}

\begin{theorem}\label{th2}Let $p>1,q\geq1$ and $f:I\subseteq\mathbb{R}\longrightarrow\mathbb{R}$
be a  differentiable  mapping  on $I^{0}$ with $0<a<b.$ If $|f^\prime|^q$ is convex and increasing on $[a,b]$, then the following inequality
\begin{equation}
\left|\frac{1}{b-a}\int_a^b f(x)dx-\frac{bf(b)-af(a)}{b-a}\right|\leq \frac{\left(|f^\prime(b)|+
|f^\prime(a)|\right)^{\frac{p-1}{p}}\left(b^{p+1}-a^{p+1}\right)^{\frac{q-1}{qp}}.\left(|b|^p|f^\prime(b)|^q+|a|^p|f^\prime(a)|^q\right)^{\frac{1}{qp}}}{2^{2-\frac{1}{p}}(p+1)^{\frac{q-1}{qp}}}.
\end{equation}
\end{theorem}
\begin{proof}The proof is parallel to that of Theorem \ref{thma} by replacing equation (\ref{KP6}) by
$$\int_a^b |t|^pdt=\frac{b^{p+1}-a^{p+1}}{p+1}.$$ We omit the further details.
\end{proof}
\begin{remark} Suppose that all the assumptions of Theorem \ref{th2}
 are satisfied with $|f^\prime|\leq M,$ and $0\leq a<b.$ we get
\begin{equation}
\left|\frac{1}{b-a}\int_a^b f(x)dx-\frac{bf(b)-af(a)}{b-a}\right|\leq \frac{M^{\frac{p-1+pq}{p}}(b^{p+1}-a^{p+1})}{2^{1/p}(b-a)}
\end{equation}
where $p>1, q\geq1.$
\end{remark}
Here, by using the classical definitions of Beta function and gamma function,
 we establish certain interesting and new inequalities are given by the next Theorems. For our purpose, We recall the Beta function and gamma function  defined by (see \cite{Sneddon})

$$B(x,y)=\int_0^1 t^{x-1}(1-t)^{y-1}dt,\;x,y>0,\;\;\textrm{and}\;\;\Gamma(x)=
\int_0^\infty t^{x-1}e^{-t}dt,\;\;x>0.$$
The Beta function satisfied the following properties:
$$B(x,x)=2^{1-2x}B(1/2,x),\;\textrm{and}\;B(x,y)=\frac{\Gamma(x)\Gamma(y)}{\Gamma(x+y)}.$$
In particular, we have
$$B(p+1,p+1)=2^{1-2(p+1)}B(1/2,p+1)=2^{1-2(p+1)}\frac{\Gamma(1/2)\Gamma(p+1)}{\Gamma(p+3/2)}
=2^{1-2(p+1)}\frac{\sqrt{\pi}\Gamma(p+1)}{\Gamma(p+3/2)}.$$

\begin{theorem}\label{th3} Let $f:I\subseteq\mathbb{R}\longrightarrow\mathbb{R}$  and suppose that $f$ has $3$ derivatives $f^\prime,\;f^{\prime\prime}$ and $f^{\prime\prime\prime}$ on $I^{0}$ with $a<b.$ If $|f^{\prime\prime}|^q$ is convex on $[a,b]$ and $|f^{\prime\prime\prime}|^q$ is convex and increasing on $[a,b]$, then the following inequality
\begin{equation}
\begin{split}
\left|\frac{1}{b-a}\int_a^b f(x)dx-\frac{bf(b)-af(a)}{b-a}-\frac{bf^\prime(b)+af^\prime(a)}{2}\right|
&\leq \frac{(b-a)^2}{2}\left(\frac{1}{6}\right)^{\frac{q-1}{q}}\Bigg\{\left[\frac{|2f^{\prime\prime}(a)|^q+|2f^{\prime\prime}(b)|^q}{12}\right]^{\frac{1}{q}}\\
&+\left[\frac{|a f^{\prime\prime\prime}(a)|^q+|bf^{\prime\prime\prime}(b)|^q}{12}\right]^{\frac{1}{q}}\Bigg\}.
\end{split}
\end{equation}
\end{theorem}
\begin{proof}
From Lemma \ref{l3}, we get
\begin{equation}\label{mp}
\frac{1}{b-a}\int_a^b xf^\prime(x)dx=\frac{af^\prime(a)+bf^\prime(b)}{2}-\frac{(b-a)^2}{2}\left[2\int_0^1t(1-t)f^{(2)}(ta+(1-t)b)dt+\int_0^1t(1-t)F(ta+(1-t)b) dt\right],
\end{equation}
where $F(t)=t f^{(3)}(t).$ From the H\"older inequality, we obtain
\begin{equation}
\begin{split}
\int_0^1t(1-t)|f^{(2)}(ta+(1-t)b)|dt&=\int_0^1[t(1-t)]^{1-\frac{1}{q}}[t(1-t)]^{\frac{1}{q}}|f^{\prime\prime}(ta+(1-t)b)|dt\\
&\leq\left[\int_0^1t(1-t)dt\right]^{\frac{q-1}{q}}\left[\int_0^1t(1-t)|f^{\prime\prime}(ta+(1-t)b)|^qdt\right]^{\frac{1}{q}}\\
&\leq \left[\int_0^1t(1-t)dt\right]^{\frac{q-1}{q}}\left[|f^{\prime\prime}(a)|^q\int_0^1t^2(1-t)dt+|f^{\prime\prime}(b)|^q\int_0^1t(1-t)^2dt\right]^{\frac{1}{q}}\\
&=\left(\frac{1}{6}\right)^{\frac{q-1}{q}}\left[\frac{|f^{\prime\prime}(a)|^q+|f^{\prime\prime}(b)|^q}{12}\right]^{\frac{1}{q}}
\end{split}
\end{equation}
Since the function $|f^{\prime\prime\prime}(t)|^q$ is convex then the function $|tf^{\prime\prime\prime}(t)|^q$ is convex as a product of two positive convex and increasing functions. So, for every $t\in[0,1]$ we have
\begin{equation}\label{malek1}
|F(ta+(1-t)b)|^q\leq t|F(a)|^q+(1-t)|F(b)|^q=t|af^{\prime\prime\prime}(a)|^q+(1-t)|bf^{\prime\prime\prime}(b)|^q.
\end{equation}
 Hence, from (\ref{malek1}) and the H\"older inequality we have
\begin{equation}
\begin{split}
\int_0^1t(1-t)|F(ta+(1-t)b)|dt&=\int_0^1[t(1-t)]^{1-\frac{1}{q}}[t(1-t)]^{\frac{1}{q}}|F(ta+(1-t)b)|dt\\
&\leq\left[\int_0^1t(1-t)dt\right]^{\frac{q-1}{q}}\left[\int_0^1t(1-t)|F(ta+(1-t)b)|^qdt\right]^{\frac{1}{q}}\\
&\leq \left[\int_0^1t(1-t)dt\right]^{\frac{q-1}{q}}\left[|F(a)|^q\int_0^1t^2(1-t)dt+|F(b)|^q\int_0^1t(1-t)^2dt\right]^{\frac{1}{q}}\\
&=\left(\frac{1}{6}\right)^{\frac{q-1}{q}}\left[\frac{|af^{\prime\prime\prime}(a)|^q+|bf^{\prime\prime\prime}(b)|^q}{12}\right]^{\frac{1}{q}}.
\end{split}
\end{equation}
So, the proof of Theorem \ref{th3} is completes.
\end{proof}

Another  similar  result  is embodied  in  the  following theorem.

\begin{theorem}\label{th6}Let $f:I\subseteq\mathbb{R}\longrightarrow\mathbb{R}$  and suppose that $f$ has $3$ derivatives $f^\prime,\;f^{\prime\prime}$ and $f^{\prime\prime\prime}$ on $I^{0}$ with $a<b.$ If $|f^{\prime\prime}|^q$ is convex on $[a,b]$ and $|f^{\prime\prime\prime}|^q$ is convex and increasing on $[a,b]$, then the following inequality
\begin{equation}\label{MP}
\begin{split}
\left|\frac{1}{b-a}\int_a^b f(x)dx-\frac{bf(b)-af(a)}{b-a}-\frac{bf^\prime(b)+af^\prime(a)}{2}\right|
&\leq \frac{(b-a)^2}{2}\left(\frac{1}{2}\right)^{1-\frac{1}{q}}\left[\frac{1}{(q+1)(q+2)(q+3)}\right]^{\frac{1}{q}}\Bigg\{\Bigg[2|2f^{\prime\prime}(a)|^q+(q+1)\\&\times|2f^{\prime\prime}(b)|^q\Bigg]^{\frac{1}{q}}+\Bigg[2|af^{\prime\prime\prime}(a)|^q+(q+1)|bf^{\prime\prime\prime}(b)|^q\Bigg]^{\frac{1}{q}}\Bigg\},
\end{split}
\end{equation}
holds for all $q\geq1.$
\end{theorem}
\begin{proof}Using the power-mean inequality, we have
\begin{equation}\label{mp2}
\begin{split}
\int_0^1t(1-t)|f^{\prime\prime}(ta+(1-t)b)|dt
&\leq\left[\int_0^1tdt\right]^{1-\frac{1}{q}}\left[\int_0^1t(1-t)^q|f^{\prime\prime}(ta+(1-t)b)|^qdt\right]^{\frac{1}{q}}\\
&\leq \left[\int_0^1tdt\right]^{1-\frac{1}{q}}\left[|f^{\prime\prime}(a)|^q\int_0^1t^2(1-t)^qdt+|f^{\prime\prime}(b)|^q\int_0^1t(1-t)^{q+1}dt\right]^{\frac{1}{q}}\\
&=\left(\frac{1}{2}\right)^{1-\frac{1}{q}}\left[\frac{2}{(q+1)(q+2)(q+3)}|f^{\prime\prime}(a)|^q+\frac{1}{(q+2)(q+3)}|f^{\prime\prime}(b)|^q\right]^{\frac{1}{q}}\\
&=\left(\frac{1}{2}\right)^{1-\frac{1}{q}}\left[\frac{1}{(q+1)(q+2)(q+3)}\right]^{\frac{1}{q}}\Bigg[2|f^{\prime\prime}(a)|^q+(q+1)|f^{\prime\prime}(b)|^q\Bigg]^{\frac{1}{q}}.
\end{split}
\end{equation}
In the same way, we get
\begin{equation}\label{mp3}
\int_0^1t(1-t)|F(ta+(1-t)b)|dt\leq \left(\frac{1}{2}\right)^{1-\frac{1}{q}}\left[\frac{1}{(q+1)(q+2)(q+3)}\right]^{\frac{1}{q}}\Bigg[2|af^{\prime\prime\prime}(a)|^q+(q+1)|bf^{\prime\prime\prime}(b)|^q\Bigg]^{\frac{1}{q}}.
\end{equation}
Combining (\ref{mp}), (\ref{mp2}) and (\ref{mp3}) we deduce that the inequality (\ref{MP}) holds.
\end{proof}
\begin{theorem}\label{th4}Let $f:I\subseteq\mathbb{R}\longrightarrow\mathbb{R}$  and suppose that $f$ has $3$ derivatives $f^\prime,\;f^{\prime\prime}$ and $f^{\prime\prime\prime}$ on $I^{0}$ with $a<b.$ If $|f^{\prime\prime}|^q$ is convex on $[a,b]$ and $|f^{\prime\prime\prime}|^q$ is convex and increasing on $[a,b]$, then the following inequality
\begin{equation}\label{Malek}
\begin{split}
\left|\frac{1}{b-a}\int_a^b f(x)dx-\frac{bf(b)-af(a)}{b-a}-\frac{bf^\prime(b)+af^\prime(a)}{2}\right|&\leq \frac{(b-a)^2}{2}\left(\frac{\sqrt{\pi}\Gamma(p+1)}{2^{1+2p}\Gamma(p+\frac{3}{2})}\right)^{\frac{1}{p}}\Bigg\{\left[\frac{|2f^{\prime\prime}(a)|^q+|2f^{\prime\prime}(b)|^q}{2}\right]^{\frac{1}{q}}\\
&+\left[\frac{|a f^{\prime\prime\prime}(a)|^q+|bf^{\prime\prime\prime}(b)|^q}{2}\right]^{\frac{1}{q}}\Bigg\},
\end{split}
\end{equation}
where $\frac{1}{p}+\frac{1}{q}=1.$
\end{theorem}
\begin{proof}Again from the H\"older inequality, we have
\begin{equation}\label{Malek2}
\begin{split}
\int_0^1t(1-t)|f^{\prime\prime}(ta+(1-t)b)|dt
&\leq\left[\int_0^1t^p(1-t)^pdt\right]^{\frac{1}{p}}\left[\int_0^1|f^{\prime\prime}(ta+(1-t)b)|^qdt\right]^{\frac{1}{q}}\\
&\leq \left[\int_0^1t^p(1-t)^pdt\right]^{\frac{1}{p}}\left[|f^{\prime\prime}(a)|^q\int_0^1tdt+|f^{\prime\prime}(b)|^q\int_0^1(1-t)dt\right]^{\frac{1}{q}}\\
&=\left(\frac{\sqrt{\pi}\Gamma(p+1)}{2^{1+2p}\Gamma(p+\frac{3}{2})}\right)^{\frac{1}{p}}\left[\frac{|f^{\prime\prime}(a)|^q+|f^{\prime\prime}(b)|^q}{2}\right]^{\frac{1}{q}}.
\end{split}
\end{equation}
In the same way we obtain
\begin{equation}\label{Malek3}
\begin{split}
\int_0^1t(1-t)|F(ta+(1-t)b)|dt
&\leq\left[\int_0^1t^p(1-t)^pdt\right]^{\frac{1}{p}}\left[\int_0^1|F(ta+(1-t)b)|^qdt\right]^{\frac{1}{q}}\\
&\leq \left[\int_0^1t^p(1-t)^pdt\right]^{\frac{1}{p}}\left[|F(a)|^q\int_0^1tdt+|F(b)|^q\int_0^1(1-t)dt\right]^{\frac{1}{q}}\\
&=\left(\frac{\sqrt{\pi}\Gamma(p+1)}{2^{1+2p}\Gamma(p+\frac{3}{2})}\right)^{\frac{1}{p}}\left[\frac{|af^{\prime\prime\prime}(a)|^q+|bf^{\prime\prime\prime}(b)|^q}{2}\right]^{\frac{1}{q}}.
\end{split}
\end{equation}
In view of (\ref{Malek2}) and (\ref{Malek3}) we deduce that the inequality (\ref{Malek}) holds true.
\end{proof}

\begin{theorem}\label{th5}Let $f:I\subseteq\mathbb{R}\longrightarrow\mathbb{R}$ and suppose that $f$ has $3$ derivatives $f^\prime,\;f^{\prime\prime}$ and $f^{\prime\prime\prime}$ on $I^{0}$ with $a<b.$ If $|f^{\prime\prime}|^q$ is convex and $|f^{\prime\prime\prime}|^q$ is convex and increasing on $[a,b]$, then the following inequality
\begin{equation}\label{Malek}
\begin{split}
\left|\frac{1}{b-a}\int_a^b f(x)dx-\frac{bf(b)-af(a)}{b-a}-\frac{bf^\prime(b)+af^\prime(a)}{2}\right|
&\leq \frac{(b-a)^2}{2}\left(\frac{1}{p+1}\right)^{\frac{1}{p}}\Bigg\{\left[\frac{|2f^{\prime\prime}(a)|^q+(q+1)|2f^{\prime\prime}(b)|^q}{(q+1)(q+2)}\right]^{\frac{1}{q}}\\&+\left[\frac{|a f^{\prime\prime\prime}(a)|^q+(q+1)|bf^{\prime\prime\prime}(b)|^q}{(q+1)(q+2)}\right]^{\frac{1}{q}}\Bigg\},
\end{split}
\end{equation}
where $\frac{1}{p}+\frac{1}{q}=1.$
\end{theorem}
\begin{proof}By using the H\"older inequality
\begin{equation}\label{Malek2}
\begin{split}
\int_0^1t(1-t)|f^{\prime\prime}(ta+(1-t)b)|dt
&\leq\left[\int_0^1t^pdt\right]^{\frac{1}{p}}\left[\int_0^1(1-t)^q|f^{\prime\prime}(ta+(1-t)b)|^qdt\right]^{\frac{1}{q}}\\
&\leq \left[\int_0^1t^pdt\right]^{\frac{1}{p}}\left[|f^{\prime\prime}(a)|^q\int_0^1t(1-t)^qdt+|f^{\prime\prime}(b)|^q\int_0^1(1-t)^{q+1}dt\right]^{\frac{1}{q}}\\
&=\left(\frac{1}{p+1}\right)^{\frac{1}{p}}\left[\frac{|f^{\prime\prime}(a)|^q+(q+1)|f^{\prime\prime}(b)|^q}{(q+1)(q+2)}\right]^{\frac{1}{q}}.
\end{split}
\end{equation}
On the other hand, we get
\begin{equation}
\begin{split}
\int_0^1t(1-t)|F(ta+(1-t)b)|dt
&\leq\left[\int_0^1t^pdt\right]^{\frac{1}{p}}\left[\int_0^1(1-t)^q|F(ta+(1-t)b)|^qdt\right]^{\frac{1}{q}}\\
&\leq \left[\int_0^1t^pdt\right]^{\frac{1}{p}}\left[|F(a)|^q\int_0^1t(1-t)^qdt+|F(b)|^q\int_0^1(1-t)^{q+1}dt\right]^{\frac{1}{q}}\\
&=\left(\frac{1}{p+1}\right)^{\frac{1}{p}}\left[\frac{|af^{\prime\prime\prime}(a)|^q+(q+1)|bf^{\prime\prime\prime}(b)|^q}{(q+1)(q+2)}\right]^{\frac{1}{q}},
\end{split}
\end{equation}
 which completes the proof.
\end{proof}

\section{Applications}
In this section, we  shall  use  the  results  of  Section 2  to  prove  by simple computation the  following  new  inequalities  connecting the  above means  for  arbitrary real numbers. \\

\noindent 1. The arithmetic mean:
$$A=A(a,b)=\frac{a+b}{2};\;a,b\in\mathbb{R}.$$
\noindent 2.  The generalized logarithmic mean:
$$L_n(a,b)=\left[\frac{b^{n+1}-a^{n+1}}{(b-a)(n+1)}\right]^{\frac{1}{n}};\;n
\in\mathbb{N},\;n\geq1,\;\;a,b\in\mathbb{R}.$$

\begin{propo} Let $a,b\in\mathbb{R},\;a<b$ and $n\in\mathbb{N}$ such that $n\geq1.$ Then the following inequality
\begin{equation}\label{88àà}
L_n^n(a,b)\leq A(|a|^n,|b|^n)
\end{equation}
\end{propo}
\begin{proof}The proof is immediate from Theorem \ref{th0} where $f(x)=x^n,\;n\geq1$.
\end{proof}
\begin{remark}We not that the inequality (\ref{88àà}) is not new,  was proved  by Agarwal and Dragomir in \cite{AD}.
\end{remark}
\begin{propo}Let $a,b\in\mathbb{R},\;a<b$ and $n\in\mathbb{N}$ such that $n\geq3.$ Then the following inequality
\begin{equation}
\left|L_n^n(a,b)\right|\leq 2n^{\frac{1-q}{pq}} \left[A\left(|a|^{n-1},|b|^{n-1}\right)\right]^{\frac{p-1}{p}}\left[A\left(|a|^{n},|b|^{n}\right)\right]^{\frac{q-1}{qp}}\left[A\left(|a|^{n(p+q)-q},|b|^{n(p+q)-q}\right)\right]^{\frac{1}{pq}}
\end{equation}
holds for all $p>1$ and $q\geq1.$
\end{propo}
\begin{proof}The  proof  is  immediate  from  Theorem \ref{thma} with $f(x)=x^n,\;x\in\mathbb{R},\;n\geq3.$
\end{proof}
\begin{propo} Let $a,b\in\mathbb{R},\;a<b$ and $n\in\mathbb{N}$ such that $n\geq3.$ Then the following inequalities
\begin{equation}\label{z1}
\left|L_n^n(a,b)+A(a^n,b^n)\right|\leq\min \left\{K_1^{(n,q)}(a,b),K_2^{(n,q)}(a,b)\right\} \frac{n(n-1)(b-a)^2}{4}
\end{equation}
and
\begin{equation}\label{z2}
\left|L_n^n(a,b)\right|\leq\min \left\{K_1^{(n,q)}(a,b),K_2^{(n,q)}(a,b)\right\} \frac{n(n-1)(b-a)^2}{8},\;a,b>0,
\end{equation}
holds true for all $q\geq1,$ where
\begin{equation*}
\begin{split}
K_1^{(n,q)}(a,b)&=\frac{1}{3}\left[A\left(|a|^{(n-2)q}, |b|^{(n-2)q}\right)\right]^{\frac{1}{q}}\\
K_2^{(n,q)}(a,b)&=\left[\frac{4}{(q+1)(q+2)(q+3)}\right]^{\frac{1}{q}}\left[A\left(2|a|^{(n-2)q}, (q+1)|b|^{(n-2)q}\right)\right]^{\frac{1}{q}}
\end{split}
\end{equation*}
\end{propo}
\begin{proof} From Theorem \ref{th3} and Theorem \ref{th6} for $f(x)=x^n,$ we obtain (\ref{z1}). Finally, combining (\ref{z1}) and (\ref{88àà}) we deduce that the inequality (\ref{z2}) holds true.
\end{proof}
\begin{remark} We note that if $q=1$, we have $K_1^{(n,1)}(a,b)=K_2^{(n,1)}(a,b),$ for all $a,b\in\mathbb{R}$ and $n\geq3.$ Consequently, we obtain that
\begin{equation}
\left|L_n^n(a,b)\right|\leq \frac{n(n-1)(b-a)^2}{24}A\left(|a|^{(n-2)}, |b|^{(n-2)} \right).
\end{equation}
\end{remark}
\begin{propo} Let $a,b\in\mathbb{R},\;a<b$ and $n\in\mathbb{N}$ such that $n\geq3.$ Then the following inequality
\begin{equation}\label{vv}
\left|L_n^n(a,b)+A(a^n,b^n)\right|\leq\min \left\{K_3^{(n,p,q)}(a,b),K_4^{(n,p,q)}(a,b)\right\} \frac{n(n-1)(b-a)^2}{2}
\end{equation}
and
\begin{equation}\label{vv1}
\left|L_n^n(a,b)\right|\leq\min \left\{K_3^{(n,p,q)}(a,b),K_4^{(n,p,q)}(a,b)\right\} \frac{n(n-1)(b-a)^2}{4},\;a,b>0,
\end{equation}
holds true for all $p,q>1,$ such that $\frac{1}{p}+\frac{1}{q}=1,$ where
\begin{equation*}
\begin{split}
K_3^{(n,p,q)}(a,b)&=\left(\frac{\sqrt{\pi}\Gamma(p+1)}{2^{1+2p}\Gamma(p+\frac{3}{2})}\right)^{\frac{1}{p}}\left[A\left(|a|^{(n-2)q}, |b|^{(n-2)q}\right)\right]^{\frac{1}{q}}\\
K_4^{(n,p,q)}(a,b)&=\left(\frac{1}{p+1}\right)^{\frac{1}{p}}\left[\frac{2}{(q+1)(q+2)}\right]^{\frac{1}{q}}\left[A\left(|a|^{(n-2)q}, (q+1)|b|^{(n-2)q}\right)\right]^{\frac{1}{q}}
\end{split}
\end{equation*}
\end{propo}
\begin{proof}The inequality (\ref{vv}) follows from Theorem \ref{th4} and Theorem \ref{th5} for $f(x)=x^n.$ Finally, the inequality (\ref{vv1}) is immediate by (\ref{88àà}) and (\ref{vv}).
\end{proof}

\end{document}